\documentclass[graybox, envcountsect,envcountsame]{svmult}

\usepackage{mathptmx}       
\usepackage{helvet}         
\usepackage{courier}        
\usepackage{type1cm}        
\usepackage{makeidx}         
\usepackage{graphicx}        
\usepackage{multicol}        
\usepackage[bottom]{footmisc}

\usepackage{amscd}
\usepackage{amsmath}
\usepackage{verbatim}
\usepackage{amssymb}
\usepackage{amsxtra}

\makeindex             

\newcommand{\onto}{\rightarrow\!\!\rightarrow}

\newcommand{\Gal}{\mathop{\mathrm{Gal}}}

\newcommand{\Ch}{\mathop{\mathrm{Ch}}\nolimits}

\newcommand{\res}{\mathop{\mathrm{res}}\nolimits}
\newcommand{\cores}{\mathop{\mathrm{cor}}\nolimits}

\newcommand{\mult}{\operatorname{mult}}

\newcommand{\Char}{\mathop{\mathrm{char}}\nolimits}

\newcommand{\Z}{\mathbb{Z}}
\newcommand{\F}{\mathbb{F}}

\newcommand{\Spec}{\operatorname{Spec}}
\newcommand{\End}{\operatorname{End}}
\newcommand{\Hom}{\operatorname{Hom}}

\newcommand{\compose}{\circ}

\newcommand{\CM}{\operatorname{CM}}
\newcommand{\C}{\operatorname{C}}

\newcommand{\corr}{\rightsquigarrow}

\renewcommand{\phi}{\varphi}

\newcommand{\RatM}{\dashrightarrow}

\renewcommand{\D}{\mathcal{D}}
\newcommand{\cT}{\mathcal{D}}

\smartqed

\begin{document}

\title*{Upper motives of outer algebraic groups}
\author{Nikita A. Karpenko}
\institute{
UPMC Univ Paris 06,
Institut de Math\'ematiques de Jussieu,
Paris,
France\\
\email{karpenko {\it at} math.jussieu.fr}
}
%
%
\maketitle

\abstract{
Let $G$ be a semisimple affine algebraic group over a field $F$.
Assuming that $G$ becomes of inner type over some finite field extension of $F$
of degree a power of a prime $p$,
we investigate the structure of the
Chow motives with coefficients in a finite field of characteristic $p$ of the projective
$G$-homogeneous varieties.
The complete motivic decomposition of any such variety contains one specific summand, which is the most
understandable among the others and which we call the {\em upper} indecomposable summand of the variety.
We show that every indecomposable motivic summand of any projective $G$-homogeneous variety is isomorphic
to a shift of the upper summand of some (other) projective $G$-homogeneous variety.
This result is already known (and has applications) in the case of $G$ of inner type and
is new for $G$ of outer type (over $F$).
\keywords
{Algebraic groups,
projective homogeneous varieties,
Chow groups and motives.\\
{\em 2000 Mathematical Subject Classifications:}
14L17; 14C25}
}

\abstract*{
Let G be a semisimple affine algebraic group over a field F.
Assuming that G becomes of inner type over some finite field extension of F
of degree a power of a prime p,
we investigate the structure of the
Chow motives with coefficients in a finite field of characteristic p of the projective
G-homogeneous varieties.
The complete motivic decomposition of any such variety contains one specific summand, which is the most
understandable among the others and which we call the upper indecomposable summand of the variety.
We show that every indecomposable motivic summand of any projective G-homogeneous variety is isomorphic
to a shift of the upper summand of some (other) projective G-homogeneous variety.
This result is already known (and has applications) in the case of G of inner type and
is new for G of outer type (over F).
}

\section
{Introduction}
\label{Introduction}

We fix an arbitrary base field $F$.
Besides of that, we fix a finite field $\F$ and we consider the
Grothendieck Chow motives over $F$
with coefficients in $\F$.
These are the objects of the category $\CM(F,\F)$, defined as in \cite{EKM}.


Let $G$ be a semisimple affine algebraic group over $F$.
According to \cite[Corollary 35(4)]{MR2264459}
(see also Corollary \ref{cor22} here), the motive of any projective $G$-homogeneous
variety decomposes (and in a unique way) into a finite direct sum of indecomposable motives.
One would like to describe the indecomposable motives which appear this way.
In this paper we do it under certain assumption on $G$ formulated in terms of the (unique up to an $F$-isomorphism)
minimal field extension
$E/F$ such that the group $G_E$ is of inner type:
the degree of $E/F$ is assumed to be a power of $p$, where $p=\Char\F$.

Note that this has been already done in \cite{upper}
in the case when $E=F$, that is, when $G$ itself is of inner type.
Therefore, though the inner case is formally included in the present paper,
we are concentrated here on the special effects of the outer case.
This remark explains the choice of the title.

Note that the extension $E/F$ is galois.
Actually, we do not use the minimality condition on the extension $E/F$ in the paper.
Therefore $E/F$ could be any finite $p$-primary galois field extension with $G_E$ of inner type.
However, it is reasonable to keep the minimality condition at least for the sake of the definition of the
{\em set of the upper motives} of $G$ which we give now.

For any intermediate field $L$ of the extension $E/F$ and any projective $G_L$-homogeneous variety $Y$,
we consider the upper (see \cite[Definition 2.10]{upper}) indecomposable summand $M_Y$ of the motive
$M(Y)\in\CM(F,\F)$ of $Y$ (considered as an $F$-variety at this point).
By definition, this is the (unique up to an isomorphism) indecomposable summand of $M(Y)$ with non-zero $0$-codimensional
Chow group.
The set of the isomorphism classes of the motives $M_Y$ for all $L$ and
$Y$ is called the set of {\em upper motives}
of the algebraic group $G$.

The summand $M_Y$ is definitely the ``easiest'' indecomposable summand of $M(Y)$ over which we have the best control.
For instance, the motive $M_Y$ is isomorphic to the motive $M_{Y'}$ for another projective homogeneous (not necessarily
under an action of the same group $G$) variety $Y'$ if and only if there exist multiplicity $1$ correspondences
$Y\corr Y'$ and $Y'\corr Y$, \cite[Corollary 2.15]{upper}.
Here a {\em correspondence} $Y\corr Y'$ is an element of the ($\dim Y$)-dimensional Chow group of
$Y\times_F Y'$ with coefficients in $\F$;
its {\em multiplicity} is its image under the push-forward to the ($\dim Y$)-dimensional Chow group of $Y$ identified with
$\F$.

One more nice property of $M_Y$ (which will be used in the proof of Theorem \ref{main}) is an easy control over the condition
that $M_Y$ is a summand of an arbitrary motive $M$: by \cite[Lemma 2.14]{upper}, this condition holds if and only if
there exist morphisms $\alpha:M(Y)\to M$ and $\beta:M\to M(Y)$ such that the correspondence
$\beta\compose\alpha$ is of multiplicity $1$.

We are going to claim that
the complete motivic decomposition of any projective $G$-homogeneous variety $X$
consists of shifts of upper motives of $G$.
In fact, the information we have is a bit more precise:

\begin{theorem}
\label{main}
For $F$, $G$, $E$, and $X$ as above, the complete motivic decomposition of $X$ consists of shifts of
upper motives of the algebraic group $G$.
More precisely, any indecomposable summand of the motive of $X$ is isomorphic
a shift of an upper motive $M_Y$
such that the Tits index of $G$ over the function field of the variety $Y$
contains the Tits index of $G$ over the function field
of $X$.
\end{theorem}

\begin{remark}
Theorem \ref{main} fails
if the degree of the extension $E/F$ is divisible by a prime different from $p$
(see Example \ref{primer}).
\end{remark}

The proof of Theorem \ref{main} is given in \S\ref{Proof of Theorem main}.
Before this, we get some preparation results which are also of independent interest.
In \S\ref{Nilpotence principle for quasi-homogeneous varieties}, we prove the nilpotence principle
for the quasi-homogeneous varieties.
In \S\ref{Motives of $0$-dimensional varieties}, we establish some properties of a motivic
corestriction functor.

By {\em sum} of motives we always mean the {\em direct} sum;
a {\em summand} is a {\em direct} summand;
a direct sum decomposition is called {\em complete} if the summands are indecomposable.

\section
{Nilpotence principle for quasi-homogeneous varieties}
\label
{Nilpotence principle for quasi-homogeneous varieties}

Let us consider the category
$\CM(F,\Lambda)$ of Grothendieck Chow motives over a field $F$ with coefficients in an
{\em arbitrary} associative commutative unital ring $\Lambda$.

We say that a smooth complete $F$-variety $X$ satisfies the nilpotence principle,
if for any $\Lambda$ and any field extension $K/F$, the kernel of the change of field homomorphism
$$
\End\big(M(X)\big)\to\End\big(M(X_K)\big)
$$
consists of nilpotents, where
$M(X)$ stands for the motive of $X$ in $\CM(F,\Lambda)$.

We say that an $F$-variety $X$ is {\em quasi-homogeneous}, if each connected component $X^0$ of $X$
has the following property: there exists a finite separable field extension $L/F$,
a semisimple affine algebraic group $G$ over $L$, and a projective $G$-homogeneous variety $Y$ such
that $Y$, considered as an $F$-variety via the composition
$Y\to\Spec L\to\Spec F$, is isomorphic to $X^0$.
(Note that the algebraic group $G$ needs not to be defined over $F$ in this definition.)

We note that any variety which is {\em projective quasi-homogeneous} in the sense of
\cite[\S4]{MR2178658} is also quasi-homogeneous in the above sense.
The following statement generalizes \cite[Theorem 8.2]{MR2110630} (see also
\cite[Theorem 5.1]{MR2178658}) and \cite[Theorem 25]{MR2264459}:

\begin{theorem}
\label{nilp}
Any quasi-homogeneous variety
satisfies the nilpotence principle.
\end{theorem}

\begin{proof}
By \cite[Theorem 92.4]{EKM} it suffices to show that the quasi-homogeneous varieties form
a {\em tractable class}.
We first recall the definition of a tractable class $\mathcal{C}$ (over $F$).
This is a disjoint union of classes $\mathcal{C}_K$ of smooth complete $K$-varieties,
where $K$ runs over {\em all} field extensions of $F$,
having the following properties:
\begin{enumerate}
\item
if $Y_1$ and $Y_2$ are in $\mathcal{C}_K$ for some $K$, then the disjoint union
of $Y_1$ and $Y_2$ is also in $\mathcal{C}_K$;
\item
if $Y$ is in $\mathcal{C}_K$ for some $K$, then each component of $Y$
is also in $\mathcal{C}_K$;
\item
if $Y$ is in $\mathcal{C}_K$ for some $K$, then for any field extension $K'/K$
the $K'$-variety $Y_{K'}$ is in $\mathcal{C}_{K'}$;
\item
if $Y$ is in $\mathcal{C}_K$ for some $K$, $Y$ is irreducible, $\dim Y>0$, and $Y(K)\ne\emptyset$,
then $\mathcal{C}_K$ contains a (not necessarily connected) variety $Y_0$ such that $\dim Y_0<\dim Y$
and $M(Y)\simeq M(Y_0)$ in  $\CM(K,\Lambda)$
(for any $\Lambda$ or, equivalently, for $\Lambda=\Z$).
\end{enumerate}

Let us define a class $\mathcal{C}$ as follows.
For any field extension $K/F$, $\mathcal{C}_K$ is the class of all
quasi-homogeneous $K$-varieties.

We claim that the class $\mathcal{C}$ is tractable.
Indeed, the properties $(1)$--$(3)$ are trivial and
the property $(4)$ is \cite[Theorem 7.2]{MR2110630}.
\qed
\end{proof}

We turn back to the case where the coefficient ring $\Lambda$ is a finite field $\F$.

\begin{corollary}
\label{cor22}
Let $M\in\CM(F,\F)$ be a summand of the motive of a quasi-homo\-geneous variety.
Then $M$ decomposes in a finite direct sum of indecomposable motives;
moreover, such a decomposition is unique (up to a permutation of the summands).
\end{corollary}

\begin{proof}
Any quasi-homogeneous variety is {\em geometrically cellular}.
In particular, it is {\em geometrically split} in the sense of \cite[\S2a]{upper}.
Finally, by Theorem \ref{nilp}, it satisfies the nilpotence principle.
The statement under proof follows now by \cite[Corollary 2.6]{upper}.
\qed
\end{proof}

\section
{Corestriction of scalars for motives}
\label{Motives of $0$-dimensional varieties}

As in the previous section, let $\Lambda$ be an arbitrary (coefficient) ring. We write $\Ch$ for the Chow group with coefficients in $\Lambda$.
Let $\C(F,\Lambda)$ be the category whose objects are pairs $(X,i)$, where
$X$ is a smooth complete equidimensional $F$-variety and $i$ is an integer.
A morphism $(X,i)\to(Y,j)$ in this category is an element of the Chow group
$\Ch_{\dim X+i-j}(X\times Y)$ (and the composition is the usual composition of correspondences).
The category $\C(F,\Lambda)$ is preadditive.
Taking first the additive completion of it, and taking then the idempotent completion of the
resulting category, one gets the category of motives $\CM(F,\Lambda)$, cf. \cite[\S63 and \S64]{EKM}.

Let $L/F$ be a finite separable field extension.
We define a functor
$$
\cores_{L/F}:\C(L,\Lambda)\to\C(F,\Lambda)
$$
as follows:
on the objects $\cores_{L/F}(X,i)=(X,i)$, where on the right-hand side $X$ is considered as an $F$-variety
via the composition $X\to\Spec L\to\Spec F$;
on the morphisms, the map
$$
\Hom_{\C(L,\Lambda)}\big((X,i),(Y,j)\big)\to\Hom_{\C(F,\Lambda)}\big((X,i),(Y,j)\big)
$$
is the push-forward homomorphism
$\Ch_{\dim X+i-j}(X\times_L Y)\to\Ch_{\dim X+i-j}(X\times_F Y)$
with respect to the closed imbedding
$X\times_LY\hookrightarrow X\times_FY$.
Passing to the additive completion and then to the idempotent completion, we get
an additive and commuting with the Tate shift functor
$\CM(L,\Lambda)\to\CM(F,\Lambda)$, which we also denote by $\cores_{L/F}$.

The functor $\cores_{L/F}:\C(L,\Lambda)\to\C(F,\Lambda)$
is left-adjoint and right-adjoint to the change of field functor
$\res_{L/F}:\C(F,\Lambda)\to\C(L,\Lambda)$, associating to $(X,i)$
the object $(X_L,i)$.
Therefore the functor $\cores_{L/F}:\CM(L,\Lambda)\to\CM(F,\Lambda)$
is also left-adjoint and right-adjoint to the change of field functor
$\res_{L/F}:\CM(F,\Lambda)\to\CM(L,\Lambda)$.
(This makes a funny difference with the category of varieties, where the functor $\cores_{L/F}$
is only left-adjoint to $\res_{L/F}$, while the right-adjoint to $\res_{L/F}$ functor is the
Weil transfer.)
It follows that for any $M\in\CM(L,\Lambda)$ and any $i\in\Z$, the Chow groups
$\Ch^i(M)$ and $\Ch^i(\cores_{L/F}M)$ are canonically isomorphic as well as the
Chow groups
$\Ch_i(M)$ and $\Ch_i(\cores_{L/F}M)$ are.
Indeed, since $\res_{L/F}\Lambda=\Lambda\in\CM(L,\Lambda)$, we have
\begin{align*}
&\Ch^i(M):=\Hom\big(M,\Lambda(i)\big)=
\Hom\big(\cores_{L/F}M,\Lambda(i)\big)
=:\Ch^i(\cores_{L/F}M)
&\text{ and}
\\
&\Ch_i(M):=\Hom\big(\Lambda(i),M\big)
=\Hom\big(\Lambda(i),\cores_{L/F}M\big)
=:\Ch_i(\cores_{L/F}M).
&
\end{align*}

In particular, if the ring $\Lambda$ is connected and $M\in\CM(L,\Lambda)$
is an {\em upper} (see \cite[Definition 2.10]{upper} or \S\ref{Introduction} here)
motivic summand of an irreducible smooth complete $L$-variety
$X$, then $\cores_{L/F}M$ is an upper motivic summand of the $F$-variety $X$.

Now we turn back to the situation where $\Lambda$ is a finite field $\F$:

\begin{proposition}
\label{three}
The following three conditions on a finite galois field extension $E/F$ are equivalent:
\begin{enumerate}
\item[$(1)$]
for any intermediate field $F\subset K\subset E$, the $K$-motive of $Spec E$ is indecomposable;
\item[$(2)$]
for any intermediate fields $F\subset K\subset L\subset E$ and any indecomposable $L$-motive $M$, the $K$-motive
$\cores_{L/K}(M)$ is indecomposable;
\item[$(3)$]
the degree of $E/F$ is a power of $p$ (where $p$ is the characteristic of the coefficient field $\F$).
\end{enumerate}
\end{proposition}

\begin{proof}
We start by showing that $(3)\Rightarrow(2)$.
So, we assume that $[E:F]$ is a power of $p$ and we prove $(2)$.
The extension $L/K$ decomposes in a finite chain of galois degree $p$ extensions.
Therefore we may assume that $L/K$ itself is a galois degree $p$ extension.
Let $R=\End(M)$.
This is an associative, unital, but not necessarily commutative $\F$-algebra.
Moreover, since $M$ is indecomposable, the ring $R$ has no non-trivial idempotents.
We have $\End\big(\cores_{L/K}(M)\big)=R\otimes_\F\End\big(M_K(\Spec L)\big)$
where $M_K(\Spec L)\in\CM(K,\F)$ is the motive of the $K$-variety $\Spec L$.
According to \cite[\S7]{MR2264459}, the ring $\End\big(M_K(\Spec L)\big)$ is isomorphic
to the group ring $\F\Gamma$, where $\Gamma$ is the Galois group of $L/K$.
Since the group $\Gamma$ is (cyclic) of order $p$, we have $\F\Gamma\simeq\F[t]/(t^p-1)$.
Since $p=\Char \F$, $\F[t]/(t^p-1)\simeq\F[t]/(t^p)$.
It follows that the ring $\End\big(\cores_{L/K}(M)\big)$ is isomorphic to the ring $R[t]/(t^p)$.
We prove $(2)$ by showing that the latter ring does not contain non-trivial idempotents.
An arbitrary element of $R[t]/(t^p)$ can be (and in a unique way) written in the form $a+b$,
where $a\in R$ and $b$ is an element of $R[t]/(t^p)$ divisible by the class of $t$.
Note that $b$ is nilpotent.
Let us take an idempotent of $R[t]/(t^p)$ and write it in the above form $a+b$.
Then $a$ is an idempotent of $R$.
Therefore $a=1$ or $a=0$.
If $a=1$, then $a+b$ is invertible and therefore $a+b=1$.
If $a=0$, then $a+b$ is nilpotent and therefore $a+b=0$.

We have proved the implication $(3)\Rightarrow(2)$.
The implication $(2)\Rightarrow(1)$ is trivial.
We finish by proving that $(1)\Rightarrow(3)$.

We assume that $[E:F]$ is divisible by a different from $p$ prime $q$ and we show that $(1)$ does not hold.
Indeed, the galois group of $E/F$ contains an element $\sigma$ of order $q$.
Let $K$ be the subfield of $E$ consisting of the elements of $E$ fixed by $\sigma$.
We have $F\subset K\subset E$ and $E/K$ is galois of degree $q$.
The endomorphisms ring of $M_K(\Spec E)$ is isomorphic to $\F[t]/(t^q-1)$.
Since $q\ne\Char\F$, the factors of the decomposition
$t^q-1=(t-1)\cdot(t^{q-1}+t^{q-2}+\dots+1)\in\F[t]$ are coprime.
Therefore the ring $F[t]/(t^q-1)$ is the direct product of the rings $\F[t]/(t-1)=\F$ and $\F[t]/(t^{q-1}+\dots+1)$,
and it follows that  the motive $M_K(\Spec E)$ is not indecomposable.
\qed
\end{proof}

\begin{corollary}
\label{upper v upper}
Let $E/F$ be a finite $p$-primary galois field extension and let $L$ be an intermediate field:
$F\subset L\subset E$.
Let $M\in\CM(L,\F)$ be an upper indecomposable motivic summand of an irreducible smooth complete $L$-variety
$X$.
Then $\cores_{L/F}M$ is an upper indecomposable summand of the $F$-variety $X$.
\qed
\end{corollary}

\begin{example}
\label{primer}
Let $X$ be a projective quadric given by an isotropic non-degenerate $4$-dimensional
quadratic form of non-trivial discriminant.
The variety $X$ is projective homogeneous under the action of the orthogonal group of the quadratic form.
This group is outer and the corresponding field extension $E/F$ of this group is the quadratic extension
given by the discriminant of the quadratic form.
The motive of $X$ contains a shift of the motive $M(\Spec E)$.

Now let us assume that the characteristic $p$ of the coefficient field $\F$ is odd.
Then $M(\Spec E)$ decomposes into a sum of two indecomposable summands.
The (total) Chow group of one of these two summands is $0$.
In particular, this summand is not an upper motive of $G$ (because the Chow group of an upper motive is non-trivial
by the very definition of upper).
Therefore Theorem \ref{main} fails without the hypothesis that the extension $E/F$ is $p$-primary.
\end{example}

\section
{Proof of Theorem \ref{main}}
\label{Proof of Theorem main}

Before starting the proof of Theorem \ref{main},
let us recall some classical facts and introduce some notation.

We write
$\D$
(or $\D_G$) for the set
of vertices of the Dynkin diagram of $G$.
The absolute galois group $\Gamma_F$ of the field $F$ acts on $\D$.
The subgroup $\Gamma_E\subset\Gamma_F$ is the kernel of the action.

Let $L$ be a field extension of $F$.
The set $\D_{G_L}$ is identified with $\D=\D_G$.
The action of $\Gamma_L$ on $\D$ is trivial if and only if the group $G_L$ is of inner type.
Any $\Gamma_L$-stable subset $\tau$ in $\D$ determines a projective $G_L$-homogeneous
variety $X_{\tau, G_L}$ in the way described in \cite[\S3]{upper}.
This is the variety corresponding to the set $\D\setminus\tau$ in the sense of \cite{MR0224710}.
For instance, $X_{\D,G_L}$ is the variety of the Borel subgroups of $G_L$, and $X_{\emptyset,G_L}=\Spec L$.
Any projective $G_L$-homogeneous variety is $G_L$-isomorphic to $X_{\tau,G_L}$ for some
$\Gamma_L$-stable $\tau\subset \D$.

If the extension $L/F$ is finite separable,
we write $M_{\tau,G_L}$ for the upper indecomposable motivic summand of the $F$-variety $X_{\tau,G_L}$
(where $\tau$ is a $\Gamma_L$-stable subset in $\D$).
If $L\subset E$, the isomorphism class of $M_{\tau,G_L}$ is an {\em upper motive of $G$}.


For any field extension $L/F$,
the set $\D_{G'}$, attached to the semisimple anisotropic kernel $G'$ of $G_L$,
is identified with a ($\Gamma_L$-invariant) subset in $\D$.
We write $\tau_L$ (or $\tau_{L,G}$) for its complement.
The subset $\tau_L\subset\D$ is (the set of circled vertices of) the Tits index
of $G_L$ defined in \cite{MR0224710}.
For any $\Gamma_L$-stable subset $\tau\subset\D$, the variety $X_{\tau,G_L}$ has a rational point if and only if
$\tau\subset\tau_L$.

\begin{proof}[of Theorem \ref{main}]
This is a recast of
\cite[proof of Theorem 3.5]{upper}.

We proof  Theorem \ref{main}
simultaneously for all $F,G,X$ using an induction on $n=\dim X$.
The base of the induction is $n=0$ where $X=\Spec F$ and the statement is trivial.

From now on we are assuming that $n\geq1$ and that Theorem
\ref{main} is already proven for all varieties of dimension $<n$.

For any field extension $L/F$, we write $\tilde{L}$ for the function field $L(X)$.

Let $M$ be an indecomposable summand of $M(X)$.
We have to show that $M$ is isomorphic to a shift of $M_{\tau,G_L}$ for some
intermediate field $L$ of $E/F$ and some $\Gal(E/L)$-stable subset
$\tau\subset \D_G$ containing $\tau_{\tilde{F}}$.

Let $G'/\tilde{F}$ be the semisimple anisotropic kernel of the group $G_{\tilde{F}}$.
The set $\D_{G'}$ is identified with $\D_G\setminus\tau_{\tilde{F},G}$.

We note that the group $G'_{\tilde{E}}$ is of inner type.
The field extension $\tilde{E}/\tilde{F}$ is galois with the galois group $\Gal(E/F)$.
In particular, its degree is a power of $p$ and
any its intermediate field is of the form $\tilde{L}$ for some intermediate field
$L$ of the extension $E/F$.

According to \cite[Theorem 4.2]{MR2178658},
the motive of $X_{\tilde{F}}$ decomposes into a sum of shifts of motives
of projective $G'_{\tilde{L}}$-homogeneous (where $L$ runs over intermediate fields of
the extension $E/F$) varieties $Y$,
satisfying $\dim Y<\dim X=n$ (we are using the assumption that $n>0$ here).
It follows by the induction hypothesis and
Corollary \ref{upper v upper},
that each summand of the complete motivic decomposition of $X_{\tilde{F}}$ is a shift of
$M_{\tau',G'_{\tilde{L}}}$ for some $L$ and some $\tau'\subset\D_{G'}$.
By Corollary \ref{cor22},
the complete decomposition of $M_{\tilde{F}}$
also consists of shifts of $M_{\tau',G'_{\tilde{L}}}$.

Let us choose a summand $M_{\tau',G'_{\tilde{L}}}(i)$ with minimal $i$ in the complete decomposition of
$M_{\tilde{F}}$.
We set $\tau=\tau'\cup\tau_{\tilde{F}}\subset\cT_G$.
We shall show that $M\simeq M_{\tau,G_L}(i)$
for these $L$, $\tau$, and $i$.

\medskip
We write $Y$ for the $F$-variety $X_{\tau,G_L}$ and we write
$Y'$ for the $\tilde{F}$-variety $X_{\tau',G'_{\tilde{L}}}$.
We write $N$ for the $F$-motive $M_{\tau,G_L}$ and we write
$N'$ for the $\tilde{F}$-motive $M_{\tau',G'_{\tilde{L}}}$.

By \cite[Lemma 2.14]{upper} (also formulated in \S\ref{Introduction} here)
and since $M$ is indecomposable,
it suffices to construct morphisms
$$
\alpha:M(Y)(i)\to M\;\;\text{and}\;\;
\beta:M\to M(Y)(i)
$$
satisfying $\mult(\beta\compose\alpha)=1$,
where $\mult(\beta\compose\alpha)$ is the {\em multiplicity}, defined in \S\ref{Introduction},
of the correspondence  $(\beta\compose\alpha)\in\Ch_{\dim Y}(Y\times_F Y)$.

We construct $\alpha$ first.
Since $\tau'\subset\tau$, the $\tilde{F}(Y)$-variety $Y'\times_{\tilde{L}}\Spec \tilde{F}(Y)$
has a rational point.
Let $\alpha_1\in\Ch_0\big(Y'\times_{\tilde{L}}\Spec \tilde{F}(Y)\big)$ be the class of a rational point.
Let $\alpha_2\in\Ch_i(X_{\tilde{F}(Y)})$ be the image of $\alpha_1$ under the composition
$$
\Ch_0\big(Y'\times_{\tilde{L}}\Spec \tilde{F}(Y)\big)\to
\Ch_0(Y'_{\tilde{F}(Y)})\to\Ch_0(N'_{\tilde{F}(Y)})\hookrightarrow\Ch_i(X_{\tilde{F}(Y)}),
$$
where the first map is the push-forward with respect to the closed imbedding
$$
Y'\times_{\tilde{L}}\Spec \tilde{F}(Y)\hookrightarrow Y'_{\tilde{F}(Y)}:=Y'\times_{\tilde{F}}\Spec \tilde{F}(Y).
$$
Since $\tau_{\tilde{F}}\subset\tau$, the variety $X$ has an $F(Y)$-point
and therefore the field extension $\tilde{F}(Y)/F(Y)$ is purely transcendental.
Consequently, the element $\alpha_2$ is $F(Y)$-rational and
lifts to an element $\alpha_3\in\Ch_{\dim Y+i}(Y\times X)$.
We mean here a lifting with respect to the composition
$$
\begin{CD}
\Ch_{\dim Y+i}(Y\times X) \onto \Ch_i(X_{F(Y)}) @>{\res_{\tilde{F}(Y)/F(Y)}}>> \Ch_i(X_{\tilde{F}(Y)})
\end{CD}
$$
where the first map is the epimorphism given by the pull-back with respect to the morphism
$X_{F(Y)}\to Y\times X$ induced by the generic point of the variety $Y$.

We define the morphism $\alpha$ as the composition
$$
\begin{CD}
M(Y)(i)@>{\alpha_3}>> M(X)@>>> M
\end{CD}
$$
where the second map is the projection of $M(X)$ onto its summand $M$.

We proceed by constructing $\beta$.
Let $\beta_1\in\Ch_{\dim Y'}(Y'\times_{\tilde{F}} Y_{\tilde{F}})$ be the class of
the closure of the graph of a rational map of $\tilde{L}$-varieties
$Y'\RatM Y_{\tilde{F}}$ (which exists because $\tau\subset\tau_{\tilde{F}}\cup \tau'$).
Note that this graph is a subset of $Y'\times_{\tilde{L}}Y_{\tilde{F}}$, which we consider as a subset
of $Y'\times_{\tilde{F}}Y_{\tilde{F}}$ via the closed imbedding
$Y'\times_{\tilde{L}}Y_{\tilde{F}}\hookrightarrow Y'\times_{\tilde{F}}Y_{\tilde{F}}$.
Let $\beta_2$ be the image of $\beta_1$ under the composition
\begin{multline*}
\Ch^{\dim Y}(Y'\times_{\tilde{F}} Y_{\tilde{F}})=
\Ch^{\dim Y}\big(M(Y')\otimes M(Y_{\tilde{F}})\big)\to
\Ch^{\dim Y}\big(N'\otimes M(Y_{\tilde{F}})\big)\to \\
\Ch^{\dim Y+i}\big(M(X_{\tilde{F}})\otimes M(Y_{\tilde{F}})\big)=
\Ch^{\dim Y+i}\big((X\times Y)_{\tilde{F}}\big)
\end{multline*}
where the first arrow is induced by the projection $M(Y')\to N'$ and the second arrow is induced
by the imbedding $N'(i)\to M(X_{\tilde{F}})$.
The element $\beta_2$ lifts to an element
$$
\beta_3\in\Ch^{\dim Y+i}(X\times X\times Y).
$$
We mean here a lifting with respect to the epimorphism
$$
\begin{CD}
\Ch^{\dim Y+i}(X\times X\times Y) \onto \Ch^{\dim Y+i}\big((X\times Y)_{\tilde{F}}\big)
\end{CD}
$$
given by the pull-back with respect to the morphism
$X\times X\times Y\to (X\times Y)_{\tilde{F}}$ induced by the generic point of the second factor in this
triple direct product.

Let $\pi\in\Ch_{\dim X}(X\times X)$ be the projector defining the
summand $M$ of $M(X)$.
Considering $\beta_3$ as a correspondence from $X$ to $X\times Y$, we define
$$
\beta_4\in\Ch^{\dim Y+i}(X\times X\times Y)
$$
as the
composition $\beta_3\compose\pi$.
We get
$$
\beta_5\in\Ch^{\dim Y+i}(X\times Y)=
\Ch_{\dim X-i}(X\times Y)
$$
as the image
of $\beta_4$ under the pull-back with respect to the diagonal
of $X$.
Finally, we define the morphism $\beta$ as the
composition
$$
\begin{CD}
M@>>>M(X)@>{\beta_5}>> M(Y)(i).
\end{CD}
$$

The verification of the relation
$\mult(\beta\compose\alpha)=1$, finishing the proof,
is similar to that of \cite[proof of Theorem 3.5]{upper}.
Since the multiplicity is not changed under extension of scalars, the computation can be done over
a splitting field of $G$.
A convenient choice is the field $\bar{F}(X)$, where $\bar{F}$ is an algebraic closure of $F$.
\qed
\end{proof}

\begin{remark}
Theorem \ref{main} can be also proven using a weaker result
in place of \cite[Theorem 4.2]{MR2178658}, namely, \cite[Theorem 7.5]{MR2110630}.
\end{remark}

\begin{acknowledgement}
Supported by the Collaborative Research Centre 701 of the Bielefeld University.
\end{acknowledgement}



\end{document}